\pdfoutput=1
\RequirePackage{ifpdf}
\ifpdf 
\documentclass[pdftex]{sigma}
\else
\documentclass{sigma}
\fi

\numberwithin{equation}{section}
\newtheorem{Theorem}{Theorem}[section]
 { \theoremstyle{definition}
\newtheorem{Example}[Theorem]{Example}
\newtheorem{Remark}[Theorem]{Remark} }

\begin{document}
\allowdisplaybreaks

\newcommand{\arXivNumber}{2006.12217}

\renewcommand{\PaperNumber}{117}

\FirstPageHeading

\ShortArticleName{A Gneiting-Like Method for Constructing Positive Definite Functions on Metric Spaces}

\ArticleName{A Gneiting-Like Method for Constructing Positive\\ Definite Functions on Metric Spaces}

\Author{Victor S.~BARBOSA~$^\dag$ and Valdir A. MENEGATTO~$^\ddag$}

\AuthorNameForHeading{V.S.~Barbosa and V.A.~Menegatto}

\Address{$^\dag$~Centro Tecnol\'{o}gico de Joinville-UFSC,\\
\hphantom{$^\dag$}~Rua Dona Francisca, 8300. Bloco U, 89219-600 Joinville SC, Brazil}
\EmailD{\href{victorrsb@gmail.com}{victorrsb@gmail.com}}

\Address{$^\ddag$~Instituto de Ci\^encias Matem\'aticas e de Computa\c{c}\~ao, Universidade de S\~ao Paulo,\\
\hphantom{$^\dag$}~Caixa Postal 668, 13560-970, S\~ao Carlos - SP, Brazil}
\EmailD{\href{menegatt@gmail.com}{menegatt@gmail.com}}

\ArticleDates{Received June 23, 2020, in final form November 07, 2020; Published online November 19, 2020}

\Abstract{This paper is concerned with the construction of positive definite functions on a~cartesian product of quasi-metric spaces using generalized Stieltjes and complete Bernstein functions. The results we prove are aligned with a well-established method of T.~Gneiting to construct space-time positive definite functions and its many extensions. Necessary and sufficient conditions for the strict positive definiteness of the models are provided when the spaces are metric.}


\Classification{42A82; 43A35}

\section{Introduction}

Let $(X,\rho)$ be a {\it quasi-metric space}, that is, a nonempty set $X$ endowed with a function $\rho\colon X\times X\allowbreak \to [0,\infty)$ (its {\it quasi-distance}) satisfying $\rho(x,x')=\rho(x',x)$ and $\rho(x,x)=0$, $x,x'\in X$.
Continuity on $(X,\rho)$ can be defined as on a metric space.
Write $D_X^\rho$ to indicate the {\it diameter-set} of $(X, \rho)$, i.e.,
\begin{gather*}
D_X^\rho=\{\rho(x,x')\colon x,x'\in X\}.
\end{gather*}
This paper is mainly concerned with {\it radial positive definite functions} on $(X,\rho)$, that is, conti\-nuous functions $f\colon D_X^\rho \to\mathbb{\mathbb{R}}$ satisfying
\begin{gather}\label{mainPD}
\sum_{j,k=1}^n c_j c_k f(\rho(x_j, x_k))\geq 0,
\end{gather}
for $n\geq 1$, reals scalars $c_1, \ldots, c_n$, and points $x_1, \ldots, x_n$ in $X$. Functions of this type play an~important role in classical analysis, approximation theory, probability theory, and statistics.
Reference \cite{wells} covers what we will need in this paper about radial positive definite functions.
The {\it strict positive definiteness} of a radial positive definite function $f$ as above demands that the inequalities be strict when the $x_j$ are distinct and the $c_j$ are not all zero. We will write $f\in {\rm PD}(X,\rho)$ and $f\in {\rm SPD}(X,\rho)$ to indicate that $f$ is positive definite and strictly positive definite on $(X,\rho)$, respectively.

The two concepts just introduced extend to a product of finitely many quasi-metric spaces. However, we will formalize the extension only in the setting to be covered in this paper. Unless stated otherwise, throughout the paper, $(X,\rho)$, $(Y,\sigma)$ and $(Z,\tau)$ will denote three quasi-metric spaces while $X\times Y \times Z$ will denote their cartesian product. Here, we will not distinguish among the spaces $X\times Y \times Z$, $X \times(Y \times Z)$ and $(X\times Y)\times Z$ and will not detach any special quasi-distance in them. A continuous function $f\colon D_X^\rho\times D_Y^\sigma\times D_Z^\tau \to \mathbb{R}$ is said to be {\it positive definite} on $X\times Y \times Z$ (the term radial will be abandoned), and we write $f\in {\rm PD}(X\times Y\times Z, \rho,\sigma,\tau)$, if
\begin{gather*}
\sum_{j,k=1}^n c_j c_k f(\rho(x_j, x_k),\sigma(y_j,y_k), \tau(z_j,z_k))\geq 0,
\end{gather*}
for $n\geq 1$, reals scalars $c_1$, \ldots, $c_n$, and points $(x_1,y_1,z_1), \ldots, (x_n,y_n,z_n)$ in $X\times Y\times Z$. A~function~$f$ in ${\rm PD}(X\times Y\times Z, \rho,\sigma,\tau)$ is {\it strictly positive definite}
if the inequalities above are strict when the $(x_j,y_j,z_j)$ are distinct and the $c_j$ are not all zero. Here we write ${\rm SPD}(X\times Y\times Z, \rho,\sigma,\tau)$.

Two problems involving the concepts of positive definiteness and strict positive definiteness are very common in the literature: to characterize ${\rm PD}(X,\rho)$, ${\rm SPD}(X,\rho)$, ${\rm PD}(X\times Y, \rho,\sigma)$, etc, for fixed choice of the spaces and to determine, explicitly, large families of functions belonging to them that have some importance in applications.

I.J.~Schoenberg characterized in \cite{schoen1} the class ${\rm PD}(\mathbb{R}^n,\rho)$ where $\rho$ is the usual Euclidean distance. His result states that a continuous function $f\colon [0,\infty) \to \mathbb{R}$ belongs to ${\rm PD}(\mathbb{R}^n,\rho)$ if and only if
\begin{gather*}
f(t)=\int_{[0,\infty)}\Omega_n(wt)\, {\rm d}\mu(w),\qquad t\geq 0,
\end{gather*}
where $\mu$ is a finite and positive measure on $[0,\infty)$ while $\Omega_n(x)=\Omega_n(\rho(x,0))$ is the mean value of $y\in S^{n-1} \mapsto {\rm e}^{{\rm i} x\cdot y}$ over $S^{n-1}$. Here, $\cdot$ denotes the usual inner product in $\mathbb{R}^n$, $S^{n-1}$ is the unit sphere in $\mathbb{R}^{n}$, if $n\geq 2$, while $S^{0}=\{-1,1\}$. He also characterized the class
${\rm PD}(\mathcal{H},\rho)$ where $\mathcal{H}$ is an infinite-dimensional Hilbert space and $\rho$ is the distance defined by its norm: a~continuous function $f\colon[0,\infty) \to \mathbb{R}$ belongs to ${\rm PD}(\mathcal{H},\rho)$ if and only if $t\in (0,\infty) \mapsto f\big(t^{1/2}\big)$ is completely monotone. Recall that a function $f\colon (0,\infty) \to \mathbb{R}$ is {\it completely monotone} if it has derivatives of all orders and $(-1)^nf^{(n)}(t)\geq 0$, for $t>0$ and $n=0,1,\ldots$. Theorem~7.14 in~\cite{wendland}
provides additional information regarding the class ${\rm PD}(\mathcal{H},\rho)$. In order to obtain the classes ${\rm SPD}(\mathbb{R}^n,\rho)$, $n\geq 2$, and ${\rm SPD}(\mathcal{H},\rho)$, one needs to eliminate the constant functions from
${\rm PD}(\mathbb{R}^n,\rho)$ and ${\rm PD}(\mathcal{H},\rho)$, respectively. Characterizations for some of the classes ${\rm PD}(L^p(A,\mu),\rho)$, where $(A,\mu)$ is a measure space and $\rho$ is given through the $p$-norm of $L^p(A,\mu)$ are presented in~\cite[Chapter~2]{wells}.

Schoenberg also provided characterizations for the classes ${\rm PD}\big(S^d,\rho\big)$, $d\geq 1$, where $\rho$ is now the geodesic distance on $S^d$. His result also included a characterization for the class ${\rm PD}(S^\infty,\rho)$, where $S^\infty$ is the unit sphere in the real Hilbert space $\ell_2$ while $\rho$ is its geodesic distance~\cite{schoen}. R.~Gangolli~\cite{gangolli} extended Schoenberg results to ${\rm PD}(H,\rho)$, where $H$ is any compact two-point homogeneous space and $\rho$ is its invariant Riemannian distance. After a normalization for the distances in these spaces is implemented, one can see that a continuous function $f\colon [0,\pi] \to \mathbb{R}$ belongs to ${\rm PD}(H,\rho)$, if and only if $f$ has a series representation in the form
\begin{gather*}
f(t)=\sum_{k=0}^\infty a_k^H P_k^H(\cos t), \qquad t \in [0,\pi],
\end{gather*}
where $a_k^H\geq 0$ for all $k$ and $\sum_{k=0}^\infty a_k^H P_k^H(1)<\infty$. Here, $P_k^H$ is the monomial $x^k$ if $H=S^\infty$ and a Jacobi polynomial of degree $k$ that depends on the space $H$ being used, otherwise. The classes ${\rm SPD}\big(S^d,\rho\big)$, ${\rm SPD}(S^\infty,\rho)$, and ${\rm SPD}(H,\rho)$ were described in~\cite{barbo,chen} through additional conditions on the sets $\big\{k\colon a_k^H>0\big\}$.

The same was done for the classes ${\rm PD}(X\times Y, \rho,\sigma)$ and ${\rm SPD}(X\times Y, \rho,\sigma)$ for some choices of $(X,\rho)$ and $(Y,\sigma)$. For the case where $X$ and $Y$ are compact two-point homogeneous spaces with their respective Riemannian distances $\rho$ and $\sigma$, the characterization for ${\rm PD}(X\times Y, \rho,\sigma)$ appeared in \cite{barbos,guella}: a continuous function $f\colon [0,\pi]^2 \to \mathbb{R}$ belongs to ${\rm SPD}(X\times Y, \rho,\sigma)$ if and only if
$f$ has a series representation in the form
\begin{gather*}
f(t,u)=\sum_{k,l=0}^\infty a_{k,l}^{X,Y} P_k^X(\cos t)P_l^Y(\cos u), \qquad t,u \in [0,\pi],
\end{gather*}
with $a_{k,l}^{X,Y}\geq 0$ for all $k$ and $l$ and the series being convergent at $(t,u)=(0,0)$. As for ${\rm SPD}(X\times Y, \rho,\sigma)$, a description can be found in \cite{barbos,guella0,guella2,guella1} and depends on additional assumptions on the sets $\big\{k-l\colon a_{k,l}^{X,Y}>0\big\}$. The cases in which $(X,\rho)$ is the usual metric space $\mathbb{R}^n$ and $Y$ is either a compact two-point homogeneous space or $S^\infty$ were considered recently: ${\rm PD}(X\times Y, \rho,\sigma)$ was described in \cite{berg2,berg3,guellam,shapiro} while a description for ${\rm SPD}(X\times Y, \rho,\sigma)$ can be inferred from~\cite{guellam}.

As for the explicit determination of large families in either ${\rm PD}(X,\rho)$ or ${\rm SPD}(X,\rho)$, the most efficient techniques make use of completely monotone functions and conditionally negative definite functions on $(X,\rho)$. A continuous function $f\colon D_X^\sigma \to \mathbb{\mathbb{R}}$ is {\it conditionally negative definite} on $(X,\rho)$, and we write $f \in {\rm CND}(X,\rho)$, if the quadratic forms in (\ref{mainPD}) are nonpositive when the coefficients $c_j$ satisfy $\sum_{j=1}^nc_j=0$. Clearly, this notion can be extended to a cartesian product of quasi-metric spaces so that the symbol ${\rm CND}(X\times Y,\rho,\sigma)$ also makes sense.

The following construction providing an efficient technique follows from Theorem~3.5 in~\cite{mene} along with Lemma~2.5 in~\cite{reams}: if $f$ is a bounded and completely monotone function and $g$ is a~nonnegative valued function in ${\rm CND}(X,\rho)$, then $f\circ g$ belongs to ${\rm PD}(X,\rho)$. Further, $f\circ g$ belongs to ${\rm SPD}(X,\rho)$ if and only if $f$ is nonconstant and $g(t)>g(0)$, for $t \in D_X^\rho\setminus\{0\}$. A~quick analysis reveals that the following extension also holds: if $f$ is a bounded and completely monotone function and $g$ is a nonnegative valued function in ${\rm CND}(X\times Y, \rho,\sigma)$, then $f\circ g$ belongs to ${\rm PD}(X \times Y,\rho,\sigma)$. Further, $f\circ g$ belongs to ${\rm SPD}(X\times Y,\rho,\sigma)$ if and only if $f$ is nonconstant and $g(t,u)>g(0,0)$, for $(t,u) \in D_X^\rho\times D_Y^\sigma$ with $t+u>0$. If we drop the boundedness of $f$, then the results above still hold as long as we assume $g$ is positive-valued.

Motivated by a celebrated result of Gneiting in \cite{gneiting}, an interesting procedure to construct positive definite functions on a cartesian product of quasi-metric spaces was described in \cite{meneg}. If $f$ is a bounded and completely monotone function, $g$ is a nonnegative valued function in~${\rm CND}(X,\rho)$ and $h$ is a positive-valued function in ${\rm CND}(Y,\sigma)$, then the function $F_r$ given by
\begin{gather} \label{menegg}
F_r(t,u)=\frac{1}{h(u)^{r}}f\left(\frac{g(t)}{h(u)}\right),\qquad (t,u)\in D_X^\rho \times D_Y^\sigma,
\end{gather}
belongs to ${\rm PD}(X \times Y,\rho,\sigma)$, as long as $f$ is a bounded generalized Stieltjes function of order $\lambda>0$~\cite{widder} and $r\geq \lambda$. Further, in the case in which $(X,\rho)$ and $(Y,\sigma)$ are metric spaces and~$X$ has at~least two points, $F_r$ belongs to ${\rm SPD}(X\times Y, \rho,\sigma)$ if and only if $f$ is nonconstant, $g(t)>g(0)$ for $t\in D_X^\rho\setminus\{0\}$, and $h(u)>h(0)$ for $u\in D_Y^\sigma\setminus\{0\}$. With some adaptations on the assumptions and specifying $r$ accordingly, similar results can be expanded to the case where $f$ is an unbounded complete monotone function.

In this paper, the target is to establish extensions of the criterion described in the previous paragraph in order to produce functions in the classes ${\rm PD}(X\times Y\times Z, \rho,\sigma,\tau)$ and ${\rm SPD}(X\times Y\times Z, \rho,\sigma,\tau)$ that can be generalized to finitely many quasi-metric spaces. From a practical point of~view, we envisage the results we will prove here to be used in random fields evolving temporally over either a torus or a cylinder. On the other hand, we also intend to prove
mathematical results that resemble some of the models discussed in \cite{allegria,apana} involving positive definiteness for the product of three metric spaces but focusing on the strict positive definiteness of the models. The outline of the paper is as follows: in Section~\ref{sec2}, we tackle the construction of conditionally negative definite functions on a product of quasi-metric spaces. They will be used in the subsequent material and are not frequently dealt with in the literature, except in~some very particular cases. We will provide two simple techniques to construct functions in~${\rm CND}(X\times Y,\rho, \sigma)$ and a third one specific for case in which $X$ is the usual metric space $\mathbb{R}^n$. In Section~\ref{sec3}, we begin describing the main contributions of the paper. We propose a model to construct strictly positive definite functions in a product of three metric spaces given by products of compositions of completely monotone functions and nonnegative valued conditionally negative definite functions. In Section~\ref{sec4}, we~focus on extensions of the model~(\ref{menegg}) to three metric spaces based on generalized Stieltjes functions of order $\lambda>0$. Section~\ref{sec5} contains adaptations of the results proved in Section~\ref{sec4} in order to produce models based on generalized complete Bernstein functions of order $\lambda >0$. In~Sec\-tion~\ref{sec6}, we address two examples that can serve as applications of the main results proved in the paper.

\section[Functions in the class CND(X x Y,rho,sigma)]{Functions in the class
$\mathbf{CND}\boldsymbol{(X\times Y,\rho,\sigma)}$}\label{sec2}

Results that deliver large classes of functions in ${\rm CND}(X\times Y,\rho,\sigma)$ are rare in the literature. Here, we will present two methods that hold in general and another one that holds in the specific case where $X$ is the usual metric space $\mathbb{R}^n$. Two of them depend upon Bernstein functions (see \cite[Chapter 3]{schilling}) the notion of which we now recall. A function $f\colon (0,\infty) \to \mathbb{R}$ is a {\it Bernstein function} if it has derivatives of all orders and $(-1)^{n-1}f^{(n)}(t)\geq 0$, for $t>0$ and $n=1,2,\ldots$. A Bernstein function $f$ has an integral representation in the form
\begin{gather*}
f(w)=a+bw+\int_{(0,\infty)} (1-{\rm e}^{-sw}) \,{\rm d}\mu(s), \qquad w\geq 0,
\end{gather*}
where $a,b\geq 0$ and $\mu$ is a positive measure on $(0,\infty)$ satisfying
\begin{gather*}
\int_{(0,\infty)}(1 \wedge s) \,{\rm d}\mu(s)<\infty.
\end{gather*}
A Bernstein function $f$ can be continuously extended to $0$ by setting $f(0)=\lim_{w\to 0^+} f(w)$. It follows from \cite[Proposition 2.9]{berg0} that if $f$ is a Bernstein function and $g$ is a nonnegative positive-valued function in ${\rm CND}(X,\rho)$, then $f\circ g$ belongs to ${\rm CND}(X,\rho)$. Theorem~\ref{bernesp}
provides a generalization of this fact.

\begin{Theorem}\label{bernesp}\sloppy
Let $f$ be a Bernstein function. If $g$ is a nonnegative valued function in ${\rm CND}(X,\rho)$ and $h$ is a nonnegative valued function in ${\rm CND}(Y,\sigma)$, then the function $\phi$ given~by
	\begin{gather*}
	\phi(t,u)=f(g(t)+h(u)),\qquad (t,u) \in D_X^\rho \times D_Y^\sigma,
	\end{gather*}
	belongs to ${\rm CND}(X\times Y, \rho,\sigma)$.
\end{Theorem}

\begin{proof} Assume $g$ and $h$ are as in the statement of the lemma. Let $n$ be a positive integer, $c_1,\ldots, c_n$ real numbers satisfying $\sum_{j=1}^n c_j=0$, and
$(x_1,y_1), \ldots, (x_n,y_n)$ points in $X\times Y$. Direct calculation shows that
\begin{gather*}
\sum_{j,k=1}^n c_j c_k f(g(\rho(x_j, x_k))+h(\sigma(y_j,y_k))) = b\sum_{j,k=1}^n c_j c_k \left[g(\rho(x_j, x_k))+ h(\sigma(y_j, y_k))\right]
\\ \hphantom{\sum_{j,k=1}^n c_j c_k f(g(\rho(x_j, x_k))+h(\sigma(y_j,y_k))) =}
{} -\int_{[0,\infty)} \sum_{j,k=1}^n c_j c_k{\rm e}^{-s g(\rho(x_j, x_k))-s h(\sigma(y_j, y_k))}\,{\rm d}\mu(s).
	\end{gather*}
	Since the function $x\in (0,\infty) \mapsto {\rm e}^{-x}$ is bounded and completely monotone and the matrix $[-s g(\rho(x_j, x_k))-s h(\sigma(y_j, y_k))]_{j,k=1}^n$ is almost positive semi-definite, then Lemma 2.5 in \cite{reams} implies that
	\begin{gather*}
	\sum_{j,k=1}^n c_j c_k {\rm e}^{- sg(\rho(x_j, x_k))-sh(\sigma(y_j,y_k))}\geq 0.
	\end{gather*}
	Thus,
	\begin{gather*}
	\sum_{j,k=1}^n c_j c_k f(g(\rho(x_j, x_k))+h(\sigma(y_j,y_k))) \leq 0,
	\end{gather*}
	and the proof is complete.
\end{proof}

Here are some examples of functions in ${\rm CND}(X\times Y, \rho,\sigma)$ provided by Theorem~\ref{bernesp} with $g$ and $h$ as there
\begin{gather*}
\phi(t,u)=f(t)+g(u), \qquad \phi(t,u)=f(t)+g(u)+\sqrt{1+f(t)+g(u)},
\\
\phi(t,u)=1-{\rm e}^{-f(t)-g(u)}, \qquad \mbox{and} \qquad \phi(t,u)=\ln(1+f(t)+g(u)).
\end{gather*}

The second method we want to present is based on positive-valued Bernstein functions and holds when one of the spaces is the usual $\mathbb{R}^n$.

\begin{Theorem} \label{auxber} Assume $\mathbb{R}^n$ is endowed with its usual Euclidean distance $\rho$.
	If $(Y,\sigma)$ is a quasi-metric space, $f$ is a positive-valued Bernstein function and $h$ is a positive-valued function in ${\rm CND}(Y,\sigma)$, then
	\begin{gather*}
	(t,u)\in [0,\infty) \times D_Y^\sigma \mapsto -\frac{1}{h(u)^{n/2}}{\rm e}^{-f(t^2/h(u))}
	\end{gather*}
	belongs to ${\rm CND}(\mathbb{R}^n \times Y, \rho, \sigma)$. Further, the formula
\begin{gather*}
	(t,u)\in [0,\infty) \times D_Y^\sigma \mapsto \frac{1}{h(0)^{n/2}}- \frac{1}{h(u)^{n/2}}{\rm e}^{- f(t^2/h(u))},\qquad w>0,
	\end{gather*}
defines a bounded positive-valued function in ${\rm CND}(\mathbb{R}^n \times Y, \rho, \sigma)$.
\end{Theorem}
\begin{proof}
	Theorem 3.7 in \cite{schilling} shows that a function $f\colon (0,\infty) \to (0,\infty)$ is a Bernstein function if and only if ${\rm e}^{-wf}$ is completely monotone for all $w>0$. So, if $f$ is a Bernstein function, then the Bernstein--Widder theorem \cite[p.~161]{widder1} leads to the representation
	\begin{gather*}
	{\rm e}^{-w f(t^2/h(u))}=\int_{[0,\infty)} {\rm e}^{-st^2/h(u)}\,{\rm d}\mu_f^w(s),\qquad (t,u)\in [0,\infty) \times D_Y^\sigma,\qquad w>0,
	\end{gather*}
	for some finite and positive measure $\mu_f^w$ on $[0,\infty)$. Since Theorem~3.2$(i)$ in~\cite{meolpo} shows that the functions
	\begin{gather*}
	(t,u)\in [0,\infty) \times D_Y^\sigma \mapsto \frac{1}{h(u)^{n/2}}{\rm e}^{-st^2/h(u)}, \qquad s>0,
	\end{gather*}
	belong to ${\rm PD}(\mathbb{R}^n \times Y, \rho, \sigma)$, we may infer that
	so do
	\begin{gather*}
	(t,u)\in [0,\infty) \times D_Y^\sigma \mapsto \frac{1}{h(u)^{n/2}}{\rm e}^{-w f(t^2/h(u))},\qquad w>0.
	\end{gather*}
	 The theorem follows after we take $w=1$.
 \end{proof}

Under the setting in Theorem~\ref{auxber}, the formula
\begin{gather*}
	(t,u)\in [0,\infty) \times D_Y^\sigma \mapsto \frac{1}{h(0)^{n/2}}- \frac{1}{h(u)^{n/2}}{\rm e}^{-w f(t^2/h(u))},\qquad w>0,
	\end{gather*}
defines bounded functions in ${\rm CND}(\mathbb{R}^n \times Y, \rho, \sigma)$.

Finally, we will provide a method to construct functions in ${\rm CND}(X\times Y, \rho,\sigma)$ via generalized Stieltjes functions. A function $f$ is a {\it generalized Stieltjes function of order} $\lambda >0$, and we will write $\mathcal{S}_\lambda$, if it can be represented in the form
\begin{gather}\label{stieltjes}
f(w)=C_f+\frac{D_f}{w^\lambda}+\int_{(0,\infty)} \frac{1}{(w+s)^\lambda}\,{\rm d}\mu_f(s),\qquad w>0,
\end{gather}
where $C_f=\lim_{w\to \infty}f(w)$, $D_f\geq 0$, and $\mu_f$ is a positive measure on $(0,\infty)$ such that
\begin{gather*}
\int_{(0,\infty)} \frac{1}{(1+s)^\lambda}\, {\rm d}\mu_f(s)<\infty.
\end{gather*}
It is not hard to see that a generalized Stieltjes function $f$ of order $\lambda $ is bounded if and only if
\begin{gather*}
D_f=0 \qquad \mbox{and} \qquad \int_{(0,\infty)}\frac{1}{s^{\lambda}}\,{\rm d}\mu_f(s)<\infty.
\end{gather*}
The set of all bounded functions from $\mathcal{S}_\lambda$ will be written as $\mathcal{S}_\lambda^b$. Examples and additional properties of functions in both $\mathcal{S}_\lambda$ and $\mathcal{S}_\lambda^b$ can be found in \cite{koum,koump,meneg,schilling,sokal} and references quoted in there. It is known that every function in $\mathcal{S}_\lambda$ is completely monotone.

\begin{Theorem} 
	Let $f$ be a function in $\mathcal{S}_\lambda^b$, $g$ a nonnegative valued function in ${\rm CND}(X,\rho)$, and $h$~a~function in ${\rm CND}(Y,\sigma)$. If the function $F_r$ in~\eqref{menegg} is bounded from above by $M>0$, then $M-F_r$ belongs to ${\rm CND}(X\times Y,\rho,\sigma)$.
\end{Theorem}
\begin{proof}
This follows from Theorem 2.4$(i)$ in \cite{meneg} where it is proved that $F_r$ belongs to \linebreak${\rm PD}(X\times Y,\rho,\sigma)$.
\end{proof}

\section[Products in PD(X x Y x Z, rho,sigma,tau)]{Products in $\mathbf{PD}\boldsymbol{(X\times Y\times Z, \rho,\sigma,\tau)}$}\label{sec3}

In this section, we will present models that may belong to either ${\rm PD}(X\times Y\times Z, \rho,\sigma,\tau)$ or~${\rm SPD}(X\times Y\times Z, \rho,\sigma,\tau)$ based upon compositions of completely monotone functions and conditionally negative definite functions. This methodology, and also the others to come in Sections~\ref{sec4} and~\ref{sec5}, presupposes the existence of conditionally negative definite functions on a product of quasi-metric spaces, reason why Section~\ref{sec2} was included here.

The Schur product theorem \cite[p.~479]{horn} implies that if $f_1$ and $f_2$ are completely monotone functions and $g$ and $h$ are positive-valued functions in ${\rm CND}(X,\rho)$ and ${\rm CND}(Y\times Z,\sigma,\tau)$, respectively, then the function $F$ given by
\begin{gather}\label{efona}
F(t,u,v)=f_1(g(t))f_2(h(u,v)),\qquad (t,u,v)\in D_X^\rho \times D_Y^\sigma \times D_Z^\tau,
\end{gather}
belongs to ${\rm PD}(X\times Y \times Z, \rho,\sigma,\tau)$. And, if $f_1$ and $f_2$ are bounded, we can even assume $g$ and $h$ are nonnegative valued. Theorem~\ref{produ} provides a setting in which the strict positive definiteness of the model can be granted.

\begin{Theorem} \label{produ} Assume $(X,\rho)$, $(Y, \sigma)$ and $(Z,\tau)$ are metric spaces. Let $f_1$ and $f_2$ be nonconstant completely monotone functions and $g$ and $h$ positive-valued functions in ${\rm CND}(X,\rho)$ and ${\rm CND}(Y\times Z,\sigma,\tau)$, respectively. The following assertions concerning the function $F$ given by~\eqref{efona} are equivalent:
	\begin{itemize}\itemsep=0pt
		\item[$(i)$] $F$ belongs to ${\rm SPD}(X\times Y \times Z, \rho,\sigma,\tau)$.
		\item[$(ii)$] $g(t)>g(0)$, for $t \in D_X^\rho\setminus\{0\}$, and $h(u,v)>h(0,0)$, for $(u,v) \in D_Y^\sigma \times D_Z^\tau$, with $u+v>0$.
	\end{itemize}
\end{Theorem}
\begin{proof}
	If $g(t)=g(0)$ for some $t\in D_X^\rho \setminus\{0\}$, we can pick two distinct points $(x_1,y_1,z_1)$ and $(x_2,y_2,z_2)$ in $X\times Y\times Z$ with
	$\rho(x_1,x_2)=t$, $y_1=y_2$, and $z_1=z_2$ in order to obtain the singular matrix
	\begin{gather*}
	\left[F(\rho(x_j,x_k),\sigma(y_j,y_k), \tau(z_j,z_k))\right]_{j,k=1}^2= \left[f_1(g(0))f_2(h(0,0))\right]_{j,k=1}^2.
	\end{gather*}
	If $h(u,v)=h(0,0)$ for some $(u,v)\in D_Y^\sigma \times D_Z^\tau$ with $u+v>0$, we can pick two distinct points $(x_1,y_1,z_1)$ and $(x_2,y_2,z_2)$ in $X\times Y\times Z$ with
	$x_1=x_2$, $\sigma(y_1,y_2)=u$, and $\tau(z_1,z_2)=v$ in order to obtain the very same singular matrix. In either case, $F$ cannot belong to ${\rm SPD}(X\times Y\times Z, \rho,\sigma,\tau)$ and the implication $(i)\Rightarrow(ii)$ follows.
	As for the converse, first we invoke the Bernstein--Widder theorem to write
	\begin{gather*}
	f_1(g(t))f_2(h(u,v)) = \int_{[0,\infty)}\left[\int_{[0,\infty)}{\rm e}^{{-g(t)s-h(u,v)s'}}\,{\rm d}\mu_{1}(s)\right]{\rm d}\mu_{2}(s'),
	\end{gather*}
	where $\mu_1$ and $\mu_2$ are (not necessarily finite) positive measures on $[0,\infty)$. Recalling the proof of~Theorem~\ref{bernesp}, we know already that the functions
	\begin{gather*}
	(t,u,v) \in D_X^\rho\times D_Y^\sigma \times D_Z^\tau \mapsto {\rm e}^{{-g(t)s-h(u,v)s'}}, \qquad s,s'>0,
	\end{gather*}
	belong to ${\rm PD}(X\times Y \times Z, \rho,\sigma, \tau)$. Hence, so do the functions
	\begin{gather}\label{double}
	(t,u,v) \in D_X^\rho\times D_Y^\sigma D_Z^\tau \mapsto \int_{[0,\infty)}{\rm e}^{{-g(t)s}}{{\rm e}^{{-h(u,v)s'}}}\,{\rm d}\mu_{1}(s), \qquad s'>0.
	\end{gather}
	If $f_2$ is nonconstant, $F$ will belong to ${\rm SPD}(X\times Y \times Z, \rho,\sigma, \tau)$ if we can show that the functions in (\ref{double})
	belong to ${\rm SPD}(X\times Y \times Z, \rho,\sigma, \tau)$. However, if $f_1$ is nonconstant, it is promptly seen that $F$ will belong to ${\rm SPD}(X\times Y \times Z, \rho,\sigma, \tau)$ as long as can show that
	the functions
	\begin{gather*}
	(t,u,v) \in D_X^\rho \times D_Y^\sigma \times D_Z^\tau \mapsto {\rm e}^{{-g(t)s-h(u,v)s'}}, \qquad s,s'>0,
	\end{gather*}
	belong to ${\rm SPD}(X\times Y \times Z, \rho,\sigma)$. So, in order to complete the proof, we will show that, under the assumptions in $(ii)$,
	the matrices
	\begin{gather*}
	\big[{\rm e}^{{-g(\rho(x_j,x_k))s - h(\sigma(y_j,y_k),\tau(z_j,z_k))s'}}\big]_{j,k=1}^n
	\end{gather*}
	are positive definite whenever $s,s'>0$ and
	$(x_1,y_1,z_1)$, \ldots, $(x_n,y_n,z_n)$ are distinct points in $X\times Y\times Z$. If $n=1$, there is nothing to be proved. If $n\geq 2$, according to Lemma 2.5 in~\cite{reams}, the aforementioned positive definiteness will hold if and only if
	\begin{gather} \label{reamv}
	g(0)s+h(0)s' < g(\rho(x_j,x_k))s+h(\sigma(y_j,y_k),\tau(z_j,z_k))s',\qquad j\neq k.
	\end{gather}
	If $x_j\neq x_k$, then $\rho(x_j,x_k)>0$ and the assumption on $g$ implies that $g(\rho(x_j,x_k))>g(0)$. If~\mbox{$y_j\neq y_k$}, then $\sigma(y_j,y_k)>0$ and the assumption on $h$ implies that
	$h(\sigma(y_j,y_k), \tau(z_j,z_k))> h(0,0)$. The same can be inferred if $z_j\neq z_k$. Thus, in any case, (\ref{reamv}) holds.
\end{proof}

The model given by (\ref{efona}) has a considerable drawback: the variables $u$ and $v$ are separated from $t$. Since separability is usually not present in models that come from applications, the results in the next sections may be interpreted as an attempt to provide models with no such inconvenience.

\section{Models based on generalized Stieltjes functions}
\label{sec4}

Here, we will extend and analyze the model (\ref{menegg}) for three quasi-metric spaces. Since there is more than one way to do this, we will begin with one possible extension of (\ref{menegg}) and will establish a basic necessary condition for its strict positive definiteness.

\begin{Theorem}\label{victor1}
	Let $f$ be a function in $\mathcal{S}_\lambda$, $g$~a~positive-valued function in ${\rm CND}(X,\rho)$ and $h$~a~posi\-tive-valued function in ${\rm CND}(Y\times Z,\sigma, \tau)$.
	For $r\geq \lambda$, set
	\begin{gather}\label{mother1}
	G_r(t,u,v)=\frac{1}{h(u,v)^r}f\left(\frac{g(t)}{h(u,v)} \right),\qquad (t,u,v)\in D_X^{\rho}\times D_{Y}^{\sigma}\times D_{Z}^{\tau}.
	\end{gather}
	The following assertions hold:
	\begin{itemize}\itemsep=0pt
		\item[$(i)$] $G_r$ belongs to ${\rm PD}(X\times Y \times Z, \rho,\sigma,\tau)$.
		\item[$(ii)$] If $G_r$ belongs to ${\rm SPD}(X\times Y \times Z, \rho,\sigma,\tau)$, then $g(t)>g(0)$, for $t \in D_X^\rho\setminus\{0\}$, and $h(u,v)>h(0,0)$, for $(u,w) \in D_Y^\sigma\times D_Z^\tau$ with $u+v>0$.
	\end{itemize}
\end{Theorem}
\begin{proof}
	Inserting the integral representation \eqref{stieltjes} for $f$ in \eqref{mother1} leads to the formula
	\begin{gather*}
	G_r(t,u,v) = \frac{C_f}{h(u,v)^r}+\frac{D_f}{g(t)^{\lambda}h(u,v)^{r-\lambda}}
	+\frac{1}{h(u,v)^{r-\lambda}}\int_{(0,\infty)} \frac{1}{[g(t)+sh(u,v)]^\lambda}\,{\rm d}\mu_f(s).
	\end{gather*}
	In order to prove $(i)$, it suffices to show that each of the three summands above belongs \linebreak to~${\rm PD}(X\times Y \times Z,\rho,\sigma,\tau)$. Once the functions
	\begin{gather*}
	w\in (0,\infty) \mapsto \frac{1}{w^{\alpha}}, \qquad \alpha=\lambda,r,r-\lambda,
	\end{gather*}
	are known to be completely monotone, some of the basic results quoted at the introduction of the paper show that $t\in D_X^\rho \mapsto g(t)^{-\lambda}$ belongs to ${\rm PD}(X,\rho)$, while
	$(u,w)\in D_Y^\sigma \times D_Z^\tau \mapsto h(u,w)^{\alpha}$, $\alpha=r,r-\lambda$, belongs to ${\rm PD}(Y\times Z, \sigma,\tau)$. Hence, it is easily seen that all the functions
	$(t,u,v)\in D_X^\rho\times D_Y^\sigma \times D_Z^\tau \mapsto g(t)^{-\lambda}$ and $(t,u,v)\in D_X^\rho\times D_Y^\sigma \times D_Z^\tau \mapsto h(u,v)^{\alpha}$, $\alpha=r,r-\lambda$, belong to ${\rm PD}(X\times Y\times Z,\rho,\sigma,\tau)$. The fact that
	${\rm PD}(X\times Y\times Z,\rho,\sigma,\tau)$ is closed under products is all that is needed in order to see that $(t,u,v)\in D_X^\rho\times D_Y^\sigma \times D_Z^\tau \mapsto g(t)^{-\lambda}h(u,v)^{r-\lambda}$ also belongs to ${\rm PD}(X\times Y\times Z,\rho,\sigma,\tau)$. It remains to show that the third summand belongs to ${\rm PD}(X\times Y\times Z,\rho,\sigma,\tau)$. Since $w\in (0,\infty) \mapsto {\rm e}^{-w}$ is completely monotone,
	the same reasoning reveals that $(t,u,v)\in D_X^\rho \times D_Y^\sigma \times D_Z^\tau \mapsto \exp(-wg(t)-wh(u,v))$ belongs to ${\rm PD}(X\times Y \times Z,\rho,\sigma,\tau)$ for $w>0$. The fact that integration with respect to an independent parameter does not affect positive definiteness and the elementary identity
	\begin{gather}\label{exp}
	\frac{\Gamma(\lambda)}{(s+t)^{\lambda}}=\int_0^\infty {\rm e}^{{-sw}}{\rm e}^{{-tw}}w^{\lambda-1}\,{\rm d}w,\qquad s,t>0,
	\end{gather}
	now imply that all the functions
	\begin{gather*}
	(t,u,v)\in D_X^\rho \times D_Y^\sigma \times D_Z^\tau \mapsto \frac{1}{[g(t)+sh(u,v)]^\lambda}, \qquad s>0,
	\end{gather*}
	belong to ${\rm PD}(X\times Y\times Z,\rho,\sigma,\tau)$. But, since ${\rm PD}(X\times Y\times Z,\rho,\sigma,\tau)$ is closed under products, we see that the remaining third summand
	\begin{gather*}
	(t,u,v)\in D_X^\rho \times D_Y^\sigma \times D_Z^\tau \mapsto \frac{1}{h(u,v)^{r-\lambda}}\int_{(0,\infty)} \frac{1}{[g(t)+sh(u,v)]^\lambda}\,{\rm d}\mu_f(s), \qquad s>0,
	\end{gather*}
	also belongs to ${\rm PD}(X\times Y\times Z,\rho,\sigma,\tau)$, completing the proof of $(i)$. If $g(t)=g(0)$, for some $t \in D_X^\rho\setminus\{0\}$, by picking two distinct points $(x_1,y_1,z_1)$ and $(x_2,y_2,z_2)$ in $X\times Y\times Z$ such that $\rho(x_1,x_2)=t$, $y_1=y_2$, and $z_1=z_2$, we obtain the singular matrix
	\begin{gather*}
	\left[G_r(\rho(x_j,x_k),\sigma(y_j,y_k), \tau(z_j,z_k))\right]_{j,k=1}^2=\left[\frac{1}{h(0,0)^r}f\left(\frac{g(0)}{h(0,0)}\right) \right]_{j,k=1}^2.
	\end{gather*}
	If $h(u,v)=h(0,0)$, for $(u,v) \in D_Y^\sigma \times D_Z^\tau$ with $u+v>0$, we can take two distinct points $(x_1,y_1,z_1)$ and $(x_2,y_2,z_2)$ in $X\times Y\times Z$ such that $x_1=x_2$, $\sigma(y_1,y_2)=u$, and $\tau(z_1,z_2)=v$ in order to obtain the very same singular matrix. In either case, we may infer that $G_r$ cannot belong to $ {\rm SPD}(X\times Y\times Z,\rho,\sigma,\tau)$. In any case, $G_r$ cannot belong to ${\rm SPD}(X\times Y \times Z, \rho,\sigma,\tau)$ and $(ii)$ follows.
\end{proof}

Henceforth, we will say a quasi-metric space is {\it nontrivial} if it contains at least two points. Theorem~\ref{victor1,5} provides additional necessary conditions for the strict positive definiteness of the model in Theorem~\ref{victor1} in some specific cases.

\begin{Theorem}\label{victor1,5}
	Let $f$ be a function in $\mathcal{S}_\lambda$, $g$ a~posi\-tive-valued function in ${\rm CND}(X,\rho)$ and $h$~a~positive-valued function in ${\rm CND}(Y\times Z,\sigma, \tau)$. The following assertion holds for the func\-tion~$G_r$ in \eqref{mother1}:
	\begin{itemize}\itemsep=0pt
		\item[$(i)$] If $(X,\rho)$ is nontrivial and $G_r$ belongs to ${\rm SPD}(X\times Y\times Z,\rho,\sigma,\tau)$, then either $D_f>0$ or~$\mu_f$ is not the zero measure.
	\end{itemize}
	Further, in the case in which $r=\lambda$ and $D_f>0$, the following additional conclusion holds:
	\begin{itemize}\itemsep=0pt
		\item[$(ii)$] If either $(Y,\sigma)$ or $(Z,\tau)$ is nontrivial and $G_\lambda$ belongs to ${\rm SPD}(X\times Y\times Z,\rho,\sigma,\tau)$, then either $C_f>0$ or $\mu_f$ is not the zero measure.
	\end{itemize}
\end{Theorem}
\begin{proof}
	If $(X,\rho)$ is nontrivial, $D_f=0$ and $\mu_f$ is the zero measure, then we can take two distinct points $(x_1,y_1,z_1)$ and $(x_2,y_2,z_2)$ in $X\times Y\times Z$ with $y_1=y_2$ and $z_1=z_2$ in order to obtain the singular matrix
	\begin{gather*}
	\left[G_r(\rho(x_j,x_k),\sigma(y_j,y_k), \tau(z_i,z_j))\right]_{j,k=1}^2=\left[\frac{C_f}{h(0,0)^r} \right]_{j,k=1}^2.
	\end{gather*}
	Similarly, if either $(Y,\sigma)$ or $(Z,\tau)$ is nontrivial, $r=\lambda$, $C_f=0<D_f$ and $\mu_f$ is the zero measure, then we can take two distinct points $(x_1,y_1,z_1)$ and $(x_2,y_2,z_2)$ in $X\times Y\times Z$ with $x_1=x_2$ in~order to obtain the singular matrix
	\begin{gather*}
	\left[G_\lambda(\rho(x_j,x_k),\sigma(y_j,y_k), \tau(z_j,z_k))\right]_{j,k=1}^2 = {\left[\frac{D_f}{g(0)^\lambda}\right]}^2_{j,k=1},
	\end{gather*}
	In either case, $G_r$ cannot belong to $ {\rm SPD}(X\times Y\times Z,\rho,\sigma,\tau)$.
\end{proof}

Theorem~\ref{victor2} achieves a necessary and sufficient condition for the strict positive definiteness of $G_r$ in the case in which $r>\lambda$ and $D_f>0$ in the representation for $f$.

\begin{Theorem} \label{victor2}
	Assume $(X,\rho)$, $(Y, \sigma)$ and $(Z,\tau)$ are metric spaces.
	Let $f$ be a function in $\mathcal{S}_\lambda$, $g$~a~positive-valued function in ${\rm CND}(X,\rho)$ and $h$ a positive-valued function in ${\rm CND}(Y\times Z,\sigma,\tau)$. If
	$D_f>0$ and $r>\lambda$, then the following assertions for $G_r$ in~\eqref{mother1} are equivalent:
	\begin{itemize}\itemsep=0pt
		\item[$(i)$] $G_r$ belongs to ${\rm SPD}(X\times Y\times Z,\rho,\sigma,\tau)$.
		\item[$(ii)$] $g(t)>g(0)$, for $t \in D_X^\rho\setminus\{0\}$, and $h(u,v)>h(0,0)$, for $(u,v) \in D_Y^\sigma\times D_Z^\tau$ with $u+v>0$.
	\end{itemize}
\end{Theorem}
\begin{proof}
	In view of Theorem~\ref{victor1}$(ii)$, only the implication $(ii)\Rightarrow(i)$ needs to be proved. Assume $D_f>0$, $r> \lambda$, and also the two assumptions on $g$ and $h$ quoted in $(ii)$.
	Theorem~\ref{produ} coupled with arguments justified in the proof of Theorem~\ref{victor1} reveal that
	\begin{gather*}
	(t,u,v) \in D_X^\rho \times D_Y^\sigma \times D_Z^\tau \mapsto \frac{D_f}{g(t)^{\lambda}h(u,v)^{r-\lambda}}
	\end{gather*}
	belongs to ${\rm SPD}(X\times Y \times Z,\rho,\sigma,\tau)$. As the other two summands appearing in the equation defining $G_r(t,u,v)$ belong to ${\rm PD}(X\times Y \times Z,\rho,\sigma,\tau)$, the result follows. \end{proof}

Next, we provide a necessary and sufficient condition for strict positive definiteness in the case in which $D_f=0$, and $r\geq\lambda$.

\begin{Theorem}\label{victor3}
	Assume $(X,\rho)$, $(Y,\sigma)$ and $(Z,\tau)$ are metric spaces. Let $f$ be a function in $\mathcal{S}_\lambda$, $g$~a~positive-valued function in ${\rm CND}(X,\rho)$ and $h$ a positive-valued function in ${\rm CND}(Y\times Z,\sigma,\tau)$. If
	$(X,\rho)$ is nontrivial, $D_f=0$, and $r\geq\lambda$, then the following assertions for $G_r$
as~\eqref{mother1} are equivalent:
	\begin{itemize}\itemsep=0pt
		\item[$(i)$] $G_r$ belongs to ${\rm SPD}(X\times Y\times Z,\rho,\sigma,\tau)$.
		\item[$(ii)$] $f$ is nonconstant, $g(t)>g(0)$, for $t \in D_X^\rho\setminus\{0\}$, and $h(u,w)>h(0,0)$, for $(u,w) \in D_Y^\sigma \times D_Z^\tau$ with $u+w>0$.
	\end{itemize}
\end{Theorem}
\begin{proof}
	If $G_r$ belongs to ${\rm SPD}(X\times Y\times Z,\rho,\sigma,\tau)$, Theorem~\ref{victor1,5}$(i)$ shows that $\mu_f$ is nonzero. In particular, $f$ is nonconstant. On the other hand, Theorem~\ref{victor1}$(ii)$ reveals that the other two conditions in $(ii)$ also hold. Thus, $(i)$ implies $(ii)$. Conversely, if $f$ is nonconstant, the assumption $D_f=0$ implies that the measure $\mu_f$ is nonzero. That being said, $(i)$ will follow if we can prove that
	\begin{gather*}
	(t,u,v) \in D_X^\rho \times D_Y^\sigma \times D_Z^\tau \mapsto \int_{(0,\infty)} \frac{1}{[g(t)+sh(u,v)]^\lambda}\,{\rm d}\mu_f(s)
	\end{gather*}
	belongs to ${\rm SPD}(X\times Y \times Z,\rho,\sigma,\tau)$ under the other two assumptions in $(ii)$. Indeed, since $(t,u,v) \in D_X^\rho \times D_Y^\sigma \times D_Z^\tau \mapsto h(u,v)^{\lambda -r}$ belongs to ${\rm PD}(X\times Y \times Z,\rho,\sigma,\tau)$ and $h(0,0)>0$, the Oppenheim--Schur inequality \cite[p.~509]{horn} will lead to~$(i)$. Since $\mu_f$ is nonzero, it suffices to show that
	\begin{gather*}
	(t,u,v) \in D_X^\rho \times D_Y^\sigma \times D_Z^\tau \mapsto \frac{1}{[g(t)+sh(u,v)]^\lambda}
	\end{gather*}
	belongs to ${\rm SPD}(X\times Y \times Z,\rho,\sigma,\tau)$, for $s>0$. By (\ref{exp}),
	what needs to be proved is that the functions
	\begin{gather*}
	(t,u,v) \in D_X^\rho \times D_Y^\sigma \times D_Z^\tau \mapsto {\rm e}^{{-g(t)w-h(u,v)sw}},\qquad w,s>0,
	\end{gather*}
	belong to ${\rm SPD}(X\times Y \times Z,\rho,\sigma,\tau)$. But that follows by the same argument employed at the end of the proof of Theorem~\ref{produ}.
\end{proof}

The proof of Theorem~\ref{victor3} justifies the following complement of Theorem~\ref{victor2}.

\begin{Theorem} 
	Assume $(X,\rho)$, $(Y,\sigma)$, and $(Z,\tau)$ are metric spaces. Let $f$ be a function in $\mathcal{S}_\lambda$, $g$~a~positive-valued function in ${\rm CND}(X,\rho)$, and $h$ a positive-valued function in ${\rm CND}(Y\times Z,\sigma,\tau)$. If
	$D_f>0$, $r=\lambda$, and $\mu_f$ is nonzero, then the following assertions for $G_r$
as~\eqref{mother1} are equivalent:
	\begin{itemize}\itemsep=0pt
		\item[$(i)$] $G_\lambda$ belongs to ${\rm SPD}(X\times Y\times Z,\rho,\sigma,\tau)$.
		\item[$(ii)$] $g(t)>g(0)$, for $t \in D_X^\rho\setminus\{0\}$, and $h(u,v)>h(0,0)$, for $(u,v) \in D_Y^\sigma\times D_Z^\tau$ with $u+w>0$.
	\end{itemize}
\end{Theorem}

It remains to consider the case in which $D_f>0$, $r=\lambda$ and $\mu_f=0$. However, Theorem~\ref{victor1,5}$(ii)$ shows that, essentially, what needs to be analyzed is the case where $C_fD_f>0$, $r=\lambda$ and~$\mu_f=0$ and also imposing the non-triviality of some of the spaces involved. In this case $G_r$ takes the form
\begin{gather*}
G_\lambda(t,u,v) = \frac{C_f}{h(u,v)^\lambda}+\frac{D_f}{g(t)^{\lambda}},\qquad (t ,u,v) \in D_X^\rho\times D_Y^\sigma\times D_Z^\tau,
\end{gather*}
with $C_fD_f>0$. So far, the strict positive definiteness of $G_r$ in this case remains an open question.

\begin{Remark}\label{rema} All the theorems proved so far can be re-stated and demonstrated for the model
	\begin{gather*}
	H_r(t,u,v)=\frac{1}{g(t)^r}f\left(\frac{h(u,v)}{g(t)} \right),\qquad (t,u,v)\in D_X^{\rho}\times D_{Y}^{\sigma}\times D_{Z}^{\tau},\qquad f\in S_\lambda,
	\end{gather*}
	with $r$, $g$ and $h$ as before. The obvious adjustments and the details on that will be left to the reader.
\end{Remark}

\section{Models based on generalized complete Bernstein functions}\label{sec5}

In this section, we will point how to extend the results proved in Section~\ref{sec4} to models defined by functions coming from the class $\mathcal{B}_\lambda$, here called the class of {\it generalized complete Bernstein functions of order $\lambda >0$}, that is, functions $f$ having a representation in the form
\begin{gather*}
f(w)=A_f+{B_f}{w^\lambda}+\int_{(0,\infty)} {\left(\frac{w}{w+s} \right)}^\lambda \,{\rm d}\nu_f(s),\qquad x>0,
\end{gather*}
where $A_f,B_f\geq 0$ and $\nu_f$ is a positive measure on $(0,\infty)$ for which
\begin{gather*}
\int_{(0,\infty)} \frac{1}{(1+s)^\lambda}\,{\rm d}\nu_f(s)<\infty.
\end{gather*}
The class $\mathcal{B}_1$ is more common in the literature. Functions in it may receive different names depending where they are used: operator monotone functions, L\"{o}wner (Loewner) functions, Pick functions, Nevanlinna functions, etc. Many examples of functions in $\mathcal{B}_\lambda$ can be found scattered in \cite{schilling}.

As we shall see below, the proofs of the results to be enunciated in this section are very similar to those of the theorems proved in Section~\ref{sec3}. For that reason, most of the details will be omitted.

We begin with a version of Theorem~\ref{victor1} for models generated by functions in $\mathcal{B}_\lambda$.

\begin{Theorem} Let $f$ be a function in $\mathcal{B}_\lambda$, $g$~a~positive-valued function in ${\rm CND}(X,\rho)$, and $h$~a~positive-valued function in ${\rm CND}(Y\times Z,\sigma, \tau)$. For $r\geq \lambda$, set
	\begin{gather}\label{bern}
	I_r(t,u,v)=\frac{1}{g(t)^r}f\left(\frac{g(t)}{h(u,v)} \right),\qquad (t,u,v)\in D_X^{\rho}\times D_{Y}^{\sigma}\times D_{Z}^{\tau}.
	\end{gather}
	The following assertions hold:
	\begin{enumerate}\itemsep=0pt
		\item[$(i)$] $I_r$ belongs to ${\rm PD}(X\times Y\times Z,\rho,\sigma,\tau)$.
		\item[$(ii)$] If $I_r$ belongs to ${\rm SPD}(X\times Y, \rho,\sigma)$, then $g(t)>g(0)$, for $t \in D_X^\rho\setminus\{0\}$, and $h(u,v)>h(0,0)$, for $(u,v) \in D_Y^\sigma\times D_Z^\tau$ with $u+v>0$.
	\end{enumerate}
\end{Theorem}
\begin{proof}
	It suffices to use the formula
	\begin{gather*}
	I_r(t,u,v)=\frac{A_f}{g(t)^r}+\frac{B_f}{h(u,v)^\lambda g(t)^{r-\lambda}}+\frac{1}{g(t)^{r-\lambda}}\int_{(0,\infty)} \frac{1}{{\left[ g(t)+sh(u,v)\right]}^\lambda} \,{\rm d}\nu_f(s)
	\end{gather*}
	that derives from the integral representation for $f$ and to mimic the proof of Theorem~\ref{victor1}.
\end{proof}

Theorem~\ref{victor1,5} takes the following form.

\begin{Theorem}
	Let $f$ be a function in $\mathcal{B}_\lambda$, $g$~a~positive-valued function in ${\rm CND}(X,\rho)$, and $h$~a~positive-valued function in ${\rm CND}(Y\times Z,\sigma, \tau)$. The following assertion holds for the func\-tion~$I_r$ in \eqref{bern}:
	\begin{itemize}\itemsep=0pt
		\item[$(i)$] If either $(Y,\sigma)$ or $(Z,\tau)$ is nontrivial and $I_r$ belongs to ${\rm SPD}(X\times Y\times Z,\rho,\sigma,\tau)$, then either $B_f>0$ or $\nu_f$ is not the zero measure.
	\end{itemize}
	In the case in which $r=\lambda$ and $D_f>0$, the following additional assumption holds:
	\begin{itemize}\itemsep=0pt
		\item[$(ii)$] If $(X,\rho)$ is nontrivial and $I_\lambda$ belongs to ${\rm SPD}(X\times Y\times Z,\rho,\sigma,\tau)$, then either $A_f>0$ or $\nu_f$ is not the zero measure.
	\end{itemize}
\end{Theorem}

As for the strict positive definiteness of the models being considered in this section, the following three results settle an if and only if condition.

\begin{Theorem}
	Assume $(X,\rho)$, $(Y,\sigma)$, and $(Z,\tau)$ are metric spaces. Let $f$ be a function in $\mathcal{B}_\lambda$,
	$g$~a~positive-valued function in ${\rm CND}(X,\rho)$, and $h$ a positive-valued function in ${\rm CND}(Y\times Z,\sigma, \tau)$.
	If $B_f>0$ and $r>\lambda$, then the following assertions for $I_r$ in~\eqref{bern} are equivalent:
	\begin{itemize}\itemsep=0pt
		\item[$(i)$] $I_r$ belongs to ${\rm SPD}(X\times Y\times Z,\rho,\sigma,\tau)$.
		\item[$(ii)$] $g(t)>g(0)$, for $t \in D_X^\rho\setminus\{0\}$, and $h(u,v)>h(0,0)$, for $(u,v) \in D_Y^\sigma\times D_Z^\tau$ with $u+v>0$.
	\end{itemize}
\end{Theorem}

\begin{Theorem}
	Assume $(X,\rho)$, $(Y,\sigma)$, and $(Z,\tau)$ are metric spaces. Let $f$ be a function in $\mathcal{B}_\lambda$,
	$g$~a~positive-valued function in ${\rm CND}(X,\rho)$, and $h$ a positive-valued function in ${\rm CND}(Y\times Z,\sigma, \tau)$.
	If either $(Y,\sigma)$ or $(Z,\tau)$ is nontrivial, $B_f=0$, and $r\geq\lambda$, then the following assertions for $I_r$ in~\eqref{bern} are equivalent:
	\begin{itemize}\itemsep=0pt
		\item[$(i)$] $I_r$ belongs to ${\rm SPD}(X\times Y\times Z,\rho,\sigma,\tau)$.
		\item[$(ii)$] $f$ is nonconstant, $g(t)>g(0)$, for $t \in D_X^\rho\setminus\{0\}$, and $h(u,v)>h(0,0)$, for $(u,v) \in D_Y^\sigma \times D_Z^\tau$, $u+v>0$.
	\end{itemize}
\end{Theorem}

\begin{Theorem}
	Assume $(X,\rho)$, $(Y,\sigma)$, and $(Z,\tau)$ are metric spaces. Let $f$ be a function in $\mathcal{B}_\lambda$,
	$g$~a~positive-valued function in ${\rm CND}(X,\rho)$, and $h$ a positive-valued function in ${\rm CND}(Y\times Z,\sigma, \tau)$. If
	$B_f>0$, $r=\lambda$, and $\nu_f$ is nonzero, then the following assertions for $I_r$ in~\eqref{bern} are equivalent:
	\begin{itemize}\itemsep=0pt
		\item[$(i)$] $I_\lambda$ belongs to ${\rm SPD}(X\times Y\times Z,\rho,\sigma,\tau)$.
		\item[$(ii)$] $g(t)>g(0)$, for $t \in D_X^\rho\setminus\{0\}$, and $h(u,v)>h(0,0)$, for $(u,v) \in D_Y^\sigma\times D_Z^\tau$ with $u+v>0$.
	\end{itemize}
\end{Theorem}

\begin{Remark} All the theorems proved so far in this section can be re-stated and proved for the model
	\begin{gather*}	
	J_r(t,u,v) = \displaystyle \frac{1}{h(u,v)^r}f\left(\frac{h(u,v)}{g(t)} \right),\qquad \displaystyle (t,u,v)\in D_X^{\rho}\times D_{Y}^{\sigma}\times D_{Z}^{\tau}, \qquad f\in \mathcal{B}_\lambda,
	\end{gather*}
	with $r$, $g$ and $h$ as before and with some small adjustments. Once again, we leave the proofs to the interested reader.
\end{Remark}

\section{Two concrete realizations}\label{sec6}

This section contains some illustrations of the theorems proved in Section~\ref{sec4}. All of them can be adapted in order to become applications of the theorems presented in Section~\ref{sec5}, but that will be left to the reader.

\begin{Example}
	Let $X$ be the unit sphere $S^d$ in $\mathbb{R}^{d+1}$ endowed with is usual geodesic distance~$\rho_d$ and let $Y=[0,\pi/2]$ and $Z=\mathbb{R}^n$
	both endowed with their usual Euclidean distances $\sigma$ and~$\tau$ respectively. The function $g$ given by the formula
	\begin{gather*}
	g(t)=3-\cos t, \qquad t\in [0,\pi],
	\end{gather*}
	belongs to ${\rm CND}\big(S^d,\rho_d\big)$ while results proved in \cite{kapil} show that, if $s\in (0,2]$, then the function $h$ given by
	\begin{gather*}
	h(u,v)=1+\sin u +v^s, \qquad (u,v) \in [0,\pi/2] \times [0,\infty),
	\end{gather*}
	belongs to ${\rm CND}(Y\times Z,\sigma,\tau)$. It is also easily seen that
	$g(t)>g(0)$ for all $t\in (0,\pi]$ and $h(u,v)>h(0,0)$ for $(u,v)\in [0,\pi/2] \times [0,\infty)$ with $u+v>0$. Under the setting of either Theorem~\ref{victor2} or Theorem~\ref{victor3}, the model 	
	\begin{gather*}
	G_r(t,u,v)=\frac{1}{{\left[1+ \sin u+v^s \right]}^r} f\left(\frac{3-\cos t }{1+\sin u +v^s}\right), \qquad (t,u,v)\in [0,\pi]\times [0,\pi/2]\times [0,\infty),
	\end{gather*}	
	defines a function $G_r$ in ${\rm SPD}(X, Y, Z, \rho_d,\sigma,\tau)$, whenever $f$ comes from $S_\lambda$. A similar conclusion holds for the model
	\begin{gather*}
	H_r(t,u,v)=\frac{1}{{\left[3-\cos t\right]}^r} f\left(\frac{1+\sin u +v^s }{3-\cos t}\right), \qquad (t,u,v)\in [0,\pi]\times [0,\pi/2]\times [0,\infty),
	\end{gather*}
	under the setting in Remark~\ref{rema}. These examples can be expanded, by letting $Z$ be a Hilbert space and $\tau$ the distance induced by its norm, keeping all the rest the same. In fact, we can let $(Z,\tau)$ be a quasi-metric space which is isometrically embedded in an infinite-dimensional Hilbert space.
\end{Example}

\begin{Example}
	Here we consider $X=\mathbb{R}$ endowed with its Euclidean norm $\rho$. On the other hand, we let $Y=S^d$ and $Z=S^{d'}$, both endowed with their geodesic distances $\sigma_d$ and $\tau_{d'}$. Since $t\in [0,\pi] \mapsto t$ belongs to both ${\rm CND}(Y,\sigma_d)$ and ${\rm CND}(Z,\tau_{d'})$, then the mapping $h\colon [0,\pi]^2 \to \mathbb{R}$ given by $h(u,v)=c+u+v$ defines a positive-valued function that belongs to ${\rm CND}(Y\times Z,\sigma_d,\tau_{d'})$, whenever $c$ is a positive constant. In addition,
	$h(u,v)>c=h(0,0)$, whenever $u+v>0$. On the other hand, $g\colon [0,\infty) \to \mathbb{R}$ given by $g(t)=t^s$, $t\geq 0$, belongs to ${\rm CND}(X,\rho)$, as long as $s\in (0,2]$. Hence, $c+g$ is a positive-valued function that belongs to ${\rm CND}(X, \rho)$ for which $g(t)>c=g(0)$ for $t>0$. With this in mind, it is now clear that under the setting of either Theorem~\ref{victor2} or~Theorem~\ref{victor3}, the model 	
	\begin{gather*}
	G_r(t,u,v)=\frac{1}{{\left[c+u+v \right]}^r} f\left(\frac{c+t^s}{c+u+v}\right), \qquad (t,u,v)\in [0,\infty) \times [0,\pi]\times [0,\pi],
	\end{gather*}	
	defines a function $G_r$ in ${\rm SPD}(X, Y, Z, \rho,\sigma_d,\tau_{d'})$, as long as $f$ comes from $S_\lambda$. The interested reader can implement considerably more complicated examples along the same lines by using the characterization of functions in ${\rm CND}\big(S^d, \sigma_d\big)$ obtained in~\cite{mene} and the many concrete examples of functions in ${\rm CND}(\mathbb{R},\rho)$ listed in~\cite{kapil}.
\end{Example}

The examples point that for the right choice of the quasi-metric spaces, the models discussed in the paper may lead to flexible, interpretable and even computationally feasible classes of~cross-covariance functions for multivariate random fields adopted in statistics. Hopefully, that will be confirmed in the near future.

\subsection*{Acknowledgements}

The authors express their gratitude to the anonymous referees for their comments and remarks which led to an improved version of the paper.

\pdfbookmark[1]{References}{ref}
\LastPageEnding


\begin{thebibliography}{99}
\footnotesize\itemsep=0pt

\bibitem{allegria}
Alegr\'{\i}a A., Porcu E., Furrer R., Mateu J., Covariance functions for
 multivariate Gaussian fields evolving temporally over planet Earth,
 \href{https://doi.org/10.1007/s00477-019-01707-w}{\textit{Stoch. Environ. Res. Risk Assess.}} \textbf{33} (2019), 1593--1608,
 \href{https://arxiv.org/abs/1701.06010}{arXiv:1701.06010}.

\bibitem{apana}
Apanasovich T., Genton M.G., Cross-covariance functions for multivariate random
 fields based on latent dimensions, \href{https://doi.org/10.1093/biomet/asp078}{\textit{Biometrika}} \textbf{97} (2010),
 15--30.

\bibitem{barbo}
Barbosa V.S., Menegatto V.A., Strictly positive definite kernels on compact
 two-point homogeneous spaces, \href{https://doi.org/10.7153/mia-19-54}{\textit{Math. Inequal. Appl.}} \textbf{19}
 (2016), 743--756, \href{https://arxiv.org/abs/1505.00591}{arXiv:1505.00591}.

\bibitem{barbos}
Barbosa V.S., Menegatto V.A., Strict positive definiteness on products of
 compact two-point homogeneous spaces, \href{https://doi.org/10.1080/10652469.2016.1249867}{\textit{Integral Transforms Spec.
 Funct.}} \textbf{28} (2017), 56--73, \href{https://arxiv.org/abs/1605.07071}{arXiv:1605.07071}.

\bibitem{berg0}
Berg C., Christensen J.P.R., Ressel P., Harmonic analysis on semigroups. Theory
 of positive definite and related functions, \href{https://doi.org/10.1007/978-1-4612-1128-0}{\textit{Graduate Texts in
 Mathematics}}, Vol.~100, Springer-Verlag, New York, 1984.

\bibitem{berg2}
Berg C., Peron A.P., Porcu E., Schoenberg's theorem for real and complex
 {H}ilbert spheres revisited, \href{https://doi.org/10.1016/j.jat.2018.02.003}{\textit{J.~Approx. Theory}} \textbf{228} (2018),
 58--78, \href{https://arxiv.org/abs/1701.07214}{arXiv:1701.07214}.

\bibitem{berg3}
Berg C., Porcu E., From {S}choenberg coefficients to {S}choenberg functions,
 \href{https://doi.org/10.1007/s00365-016-9323-9}{\textit{Constr. Approx.}} \textbf{45} (2017), 217--241, \href{https://arxiv.org/abs/1505.05682}{arXiv:1505.05682}.

\bibitem{chen}
Chen D., Menegatto V.A., Sun X., A necessary and sufficient condition for
 strictly positive definite functions on spheres, \href{https://doi.org/10.1090/S0002-9939-03-06730-3}{\textit{Proc. Amer. Math.
 Soc.}} \textbf{131} (2003), 2733--2740.

\bibitem{gangolli}
Gangolli R., Positive definite kernels on homogeneous spaces and certain
 stochastic processes related to {L}\'evy's {B}rownian motion of several
 parameters, \textit{Ann. Inst. H.~Poincar\'e Sect.~B (N.S.)} \textbf{3}
 (1967), 121--226.

\bibitem{gneiting}
Gneiting T., Nonseparable, stationary covariance functions for space-time data,
 \href{https://doi.org/10.1198/016214502760047113}{\textit{J.~Amer. Statist. Assoc.}} \textbf{97} (2002), 590--600.

\bibitem{guella0}
Guella J.C., Menegatto V.A., Strictly positive definite kernels on a product of
 spheres, \href{https://doi.org/10.1016/j.jmaa.2015.10.026}{\textit{J.~Math. Anal. Appl.}} \textbf{435} (2016), 286--301.

\bibitem{guellam}
Guella J.C., Menegatto V.A., Schoenberg's theorem for positive definite
 functions on products: a unifying framework, \href{https://doi.org/10.1007/s00041-018-9631-5}{\textit{J.~Fourier Anal. Appl.}}
 \textbf{25} (2019), 1424--1446.

\bibitem{guella}
Guella J.C., Menegatto V.A., Peron A.P., An extension of a theorem of
 {S}choenberg to products of spheres, \href{https://doi.org/10.1215/17358787-3649260}{\textit{Banach~J. Math. Anal.}}
 \textbf{10} (2016), 671--685, \href{https://arxiv.org/abs/1503.08174}{arXiv:1503.08174}.

\bibitem{guella2}
Guella J.C., Menegatto V.A., Peron A.P., Strictly positive definite kernels on
 a product of spheres~{II}, \href{https://doi.org/10.3842/SIGMA.2016.103}{\textit{SIGMA}} \textbf{12} (2016), 103, 15~pages,
 \href{https://arxiv.org/abs/1605.09775}{arXiv:1605.09775}.

\bibitem{guella1}
Guella J.C., Menegatto V.A., Peron A.P., Strictly positive definite kernels on
 a product of circles, \href{https://doi.org/10.1007/s11117-016-0425-1}{\textit{Positivity}} \textbf{21} (2017), 329--342,
 \href{https://arxiv.org/abs/1505.01169}{arXiv:1505.01169}.

\bibitem{horn}
Horn R.A., Johnson C.R., Matrix analysis, 2nd~ed., \href{https://doi.org/10.1017/9781139020411}{Cambridge University Press},
 Cambridge, 2013.

\bibitem{kapil}
Kapil Y., Pal R., Aggarwal A., Singh M., Conditionally negative definite
 functions, \href{https://doi.org/10.1007/s00009-018-1239-0}{\textit{Mediterr.~J. Math.}} \textbf{15} (2018), 199, 12~pages.

\bibitem{koum}
Koumandos S., Pedersen H.L., On asymptotic expansions of generalized
 {S}tieltjes functions, \href{https://doi.org/10.1007/s40315-014-0094-7}{\textit{Comput. Methods Funct. Theory}} \textbf{15}
 (2015), 93--115.

\bibitem{koump}
Koumandos S., Pedersen H.L., On generalized {S}tieltjes functions,
 \href{https://doi.org/10.1007/s00365-018-9440-8}{\textit{Constr. Approx.}} \textbf{50} (2019), 129--144, \href{https://arxiv.org/abs/1706.00606}{arXiv:1706.00606}.

\bibitem{mene}
Menegatto V.A., Strictly positive definite kernels on the {H}ilbert sphere,
 \href{https://doi.org/10.1080/00036819408840292}{\textit{Appl. Anal.}} \textbf{55} (1994), 91--101.

\bibitem{meneg}
Menegatto V.A., Positive definite functions on products of metric spaces via
 generalized {S}tieltjes functions, \href{https://doi.org/10.1090/proc/15137}{\textit{Proc. Amer. Math. Soc.}}
 \textbf{148} (2020), 4781--4795.

\bibitem{meolpo}
Menegatto V.A., Oliveira C., Porcu E., Gneiting class, semi-metric spaces, and
 isometric embeddings, \href{https://doi.org/10.33205/cma.712049}{\textit{Constr. Math. Anal.}} \textbf{3} (2020), 85--95.

\bibitem{reams}
Reams R., Hadamard inverses, square roots and products of almost semidefinite
 matrices, \href{https://doi.org/10.1016/S0024-3795(98)10162-3}{\textit{Linear Algebra Appl.}} \textbf{288} (1999), 35--43.

\bibitem{schilling}
Schilling R.L., Song R., Vondra\v{c}ek Z., Bernstein functions. Theory and applications, 2nd ed., \textit{De
 Gruyter Studies in Mathematics}, Vol.~37, \href{https://doi.org/10.1515/9783110269338}{Walter de Gruyter \& Co.},
 Berlin, 2012.

\bibitem{schoen1}
Schoenberg I.J., Metric spaces and completely monotone functions, \href{https://doi.org/10.2307/1968466}{\textit{Ann.
 of Math.}} \textbf{39} (1938), 811--841.

\bibitem{schoen}
Schoenberg I.J., Positive definite functions on spheres, \href{https://doi.org/10.1215/S0012-7094-42-00908-6}{\textit{Duke Math.~J.}}
 \textbf{9} (1942), 96--108.

\bibitem{shapiro}
Shapiro V.L., Fourier series in several variables with applications to partial
 differential equations, \textit{Chapman \& Hall/CRC Applied Mathematics and Nonlinear
 Science Series}, \href{https://doi.org/10.1201/b10811}{CRC Press}, Boca Raton, FL, 2011.

\bibitem{sokal}
Sokal A.D., Real-variables characterization of generalized {S}tieltjes
 functions, \href{https://doi.org/10.1016/j.exmath.2009.06.004}{\textit{Expo. Math.}} \textbf{28} (2010), 179--185,
 \href{https://arxiv.org/abs/0902.0065}{arXiv:0902.0065}.

\bibitem{wells}
Wells J.H., Williams L.R., Embeddings and extensions in analysis,
 \textit{Ergebnisse der Mathematik und ihrer Grenzgebiete}, Vol.~84,
 \href{https://doi.org/10.1007/978-3-642-66037-5}{Springer-Verlag}, New York~-- Heidelberg, 1975.

\bibitem{wendland}
Wendland H., Scattered data approximation, \textit{Cambridge Monographs on
 Applied and Computational Mathe\-matics}, Vol.~17, \href{https://doi.org/10.1017/CBO9780511617539}{Cambridge University Press},
 Cambridge, 2005.

\bibitem{widder}
Widder D.V., The {S}tieltjes transform, \href{https://doi.org/10.2307/1989901}{\textit{Trans. Amer. Math. Soc.}}
 \textbf{43} (1938), 7--60.

\bibitem{widder1}
Widder D.V., The {L}aplace {T}ransform, \textit{Princeton Mathematical Series},
 Vol.~6, Princeton University Press, Princeton, N. J., 1941.

\end{thebibliography}
\end{document}